\documentclass{amsart}
\usepackage{authblk}
\usepackage{setspace}
\usepackage[margin=1.25in]{geometry}
\usepackage{graphicx}
\graphicspath{ {./figures/} }
\usepackage{subcaption}
\usepackage{amsmath}

\usepackage{changepage}

\usepackage[english]{babel}

\usepackage{framed}
\usepackage[normalem]{ulem}

\usepackage{amsthm}
\usepackage{amssymb}
\usepackage{amsfonts}
\usepackage{enumerate}
\usepackage{enumitem}
\usepackage{tabto}
\usepackage{xcolor}
\usepackage{tikz}
\usepackage[utf8]{inputenc}

\usepackage{float}
\usepackage{hyperref}

\theoremstyle{definition}
\newtheorem{theorem}{Theorem}
\newtheorem{proposition}{Proposition}
\newtheorem{lemma}{Lemma}

\newtheorem{definition}{Definition}[section]
\newtheorem{example}{Example}

\newtheorem{construction}{Construction}

\title{Non-uniform Lattices of Large Systole 
Containing a Fixed 3-Manifold Group
}

\author{Paige Hillen}

\affil{University of California, Santa Barbara}

\date{June 2024}

\begin{document}

\maketitle
\vspace{-0.3cm}
\begin{center} \begin{small}
    PAIGE HILLEN 
    \end{small}
\end{center}

\begin{abstract}
     Let $d\geq 2$ be a square free integer and $\mathbb{Q}(\sqrt{d})$ a totally real quadratic field over $\mathbb{Q}$. We show there exists an arithmetic lattice $\mathcal{L}$ in $SL(8,\mathbb{R})$ with entries in the ring of integers of $\mathbb{Q}(\sqrt{d})$ and a sequence of lattices $\Lambda_n $ commensurable to $\mathcal{L}$ such that the systole of the locally symmetric finite volume manifold $$\Lambda_n \diagdown SL(8,\mathbb{R}) \diagup SO(8)$$ goes to infinity as $n \rightarrow \infty$, yet every $\Lambda_n$ contains the same hyperbolic 3-manifold group $\Pi$, a finite index subgroup of the arithmetic hyperbolic 3-manifold vol3. Notably, such an example does not exist in rank one, so this is a feature unique to higher rank lattices.  
\end{abstract}

\section{Introduction}

The relationship between the the \textit{systolic genus}, the minimal genus surface subgroup, and the \textit{systole}, the minimal length of a non-contractible closed geodesic, is notably different in higher rank. In 2012, Belolipetsky \cite{Bel} (see also \cite{BD}) showed that in the hyperbolic setting, the systolic genus is bounded from below in terms of the systole. In contrast, Long and Reid \cite{Long1} found a family of sequences of lattices in $SL(3,\mathbb{R})$ (commensurable to an arbitrary non-uniform arithmetic lattice not commensurable to $SL(3,\mathbb{Z})$) with systole going to infinity, yet each contains a genus 3 surface subgroup. They found a similar result in the uniform case. Hence in higher rank, the systolic genus is not linked to the systole in the same way as in the hyperbolic setting. 

Our result continues in this line of research. We expand Long and Reid's result to the existence of a fixed 3-manifold group in a sequence of commensurable lattices with arbitrarily large systole. More specifically, we find an infinite family of non-uniform arithmetic lattices in $SL(8,\mathbb{R})$ each with a sequence of commensurable lattices whose systole $\rightarrow \infty$, however every lattice in the sequence contains the same hyperbolic 3-manifold group. Our non-uniform arithmetic lattices are indexed by square free numbers $ d \geq 2$: for each such $d$, we consider the integral special unitary group
$$SU(I_{8};\mathcal{O}_d,\tau) := \{A \in SL(8,\mathcal{O}_d): \tau(A)^\top A=I_{8}\} < SL(8,\mathbb{R}) $$

\noindent where $I_8 =$ the identity matrix in $SL(8,\mathbb{R})$, $\mathcal{O}_d$ is the ring of integers of $\mathbb{Q}(\sqrt{d})$ and $\tau \in \text{Gal}(\mathbb{Q}(\sqrt{d})/\mathbb{Q})$ is the non-trivial involution sending $\sqrt{d}$ to $-\sqrt{d}$.

\begin{theorem}\label{main theorem}
    \textit{Fix square free $d \in \mathbb{Z}_{\geq 2}$ and let
    $$\mathcal{L} := \text{SU}(I_{8};\mathcal{O}_d,\tau). $$
    There exists a sequence of non-uniform arithmetic lattices $\Lambda_n < SL(8,\mathbb{R})$ commensurable to $\mathcal{L}$ such that 
    $$\text{sys}(\Lambda_n) \rightarrow \infty$$
    as $n \rightarrow \infty$, yet every $\Lambda_n$ contains a fixed hyperbolic 3-manifold group $\Pi$, a finite index subgroup of vol3.} 
\end{theorem}

 In this setting, rank and determinant up to $\tau$-Hermitian square classify SU-equivalent $\tau$-Hermitian forms  \cite{Landherr1}.  In turn, equivalent $\tau$-Hermitian forms yield commensurable integral special unitary groups. Hence $\text{SU}(I_{8};\mathcal{O}_d,\tau)$ is commensurable to $\text{SU}(J;\mathcal{O}_d,\tau)$ for any non-degenerate $\tau$-Hermitian form $J$ over $\mathbb{Q}(\sqrt{d})^{8}$ such that det$(J)$ is a $\tau$-Hermitian square.

Critical to the proof is an 8 dimensional version of the family of discrete and faithful representations $\rho_t$ of the hyperbolic 3-manifold vol3 found by Cooper, Long, and Thistlethwaite \cite{Cooper3} (see also \cite{Cooper1}). We were unable to choose values of $t$ for which the entries of $\rho_t$ lie in a ring of integers. However we are able to choose such specialized values of $t$ for a certain conjugate of $\rho_t \oplus \rho_t$. From this family, we obtain a sequence of representations of vol3 into lattices $SU(J;\mathcal{O}_d,\tau)$ for non-degenerate forms $J$. To ensure the systole is going to infinity, we then consider principal congruence subgroups of each $SU(J;\mathcal{O}_d,\tau)$ of level $p$ for an increasing sequence of primes $p$. We can no longer guarantee vol3 is contained in principal congruence subgroup, but by carefully choosing the primes $p$, we can guarantee the fundamental group of a certain fixed 320-sheeted cover of vol3 is still contained in each principal congruence subgroup.

 It is worth remarking that our finite index subgroup of vol3, $\Pi$, contains surface subgroups. Therefore our result provides an alternative proof to Long and Reid's \cite{Long1} result that systolic genus can be bounded in a sequence of commensurable higher rank lattices whose systole is diverging to infinity. \\

\noindent \textbf{Acknowledgements.} The author would like to thank Darren Long for his steady guidance and cautious optimism throughout the development of this work. Additionally, the author is grateful to the referee for helpful comments on a previous version of this paper. 

\bigskip 

\section{Background}

\noindent Consider the following measures of the geometry and topology of a space: 
\begin{definition}
      The \textit{systole} of a Riemannian manifold $M$, denoted sys$(M)$, is the minimal length of a non-contractible closed geodesic in $M$. We will interchangeably refer to the systole of $M$ as the systole of $\pi_1(M)$. 
\end{definition}

\begin{definition} Let $S_g$ denote the closed surface with genus $g \geq 2$. The minimal $g$ such that $\pi_1(S_g)$ injects into $\pi_1(M)$ is called the \textit{systolic genus of $M$}, denoted sysg$(M)$ (See \cite{Bel}). 
\end{definition}

For hyperbolic manifolds, the behavior of the systole puts some restrictions on the behavior of the systolic genus. This is clear for hyperbolic surfaces, since Besicovitch's inequality 
$$\text{sys}(S_g)^2 \leq 2 \text{area}(S_g)$$
combined with Gauss-Bonnet theorem (for a hyperbolic metric on $S_g$),
$$\text{area}(S_g) \leq 4\pi(g-1)$$
together show that a sequence of closed hyperbolic surfaces with systole $\rightarrow \infty$ requires the genera of the surfaces also tends toward infinity. In fact, Belolipetsky showed this generalizes to higher dimensions in the following sense (see Theorem 5.1 in \cite{Bel}): 

\begin{theorem}\label{belolipetsky}
    \textit{Let $M_n = \mathbb{H}^m/\Gamma_n$ be a sequence of closed hyperbolic $m$-manifolds
    such that sys$(\Gamma_n) \rightarrow \infty$ as $n \rightarrow \infty$. Then the systolic genus sysg$(M_n) \rightarrow \infty$ as well.} 
\end{theorem}

Intuitively, if the manifolds are getting complicated enough for the systole to grow arbitrarily large, their topology must be getting complicated as well. However, once we leave the hyperbolic setting, this is no longer true. In particular, the presence of flats seems to allow enough space for systoles to grow, without the systolic genus increasing. 

\begin{definition}
    A discrete subgroup $\Gamma < SL(m,\mathbb{R})$ is a \textit{lattice} if the quotient orbifold 
    $$M_{\Gamma} := \Gamma \backslash SL(m,\mathbb{R})/ SO(m)  $$
    has finite volume. Note that $M_{\Gamma}$ is a manifold if and only if $\Gamma$ is torsion free. A lattice $\Gamma$ is \textit{uniform} (or cocompact) if $M_\Gamma$ is compact. Otherwise, $\Gamma$ is \textit{non-uniform}. 
\end{definition}

In the theory of lattices, passing to a finite index subgroup usually results in only minor differences. Since we often like to ignore these minor differences, we usually care about lattices up to commensurability. 

\begin{definition}
    Lattices $\Gamma_1, \Gamma_2 < SL(m,\mathbb{R})$ are \textit{commensurable} if for some $g \in SL(m,\mathbb{R})$
    \begin{center}
        $[ \Gamma_1: \Gamma_1 \cap g\Gamma_2 g^{-1}] <\infty$ 
    \end{center}
    Equivalently, their corresponding manifolds $M_{\Gamma_1}$ and $M_{\Gamma_2}$ have a common finite-sheeted cover (i.e. $M_{\Gamma_1 \cap g\Gamma_2 g^{-1}})$. 
\end{definition} 

We review a construction of a family of non-uniform arithmetic lattices in $SL(m,\mathbb{R})$. See \cite{Witte} Chapter 6.8 for more details. 
 
\begin{construction}\label{constuction 1}
Fix a square-free $d \in \mathbb{Z}_{>1}$ and $m \geq 3$. Then
\begin{itemize}
    \item $F= \mathbb{Q}(\sqrt{d})$ is a totally real algebraic number field,
    \item $\tau\in \text{Gal}(F/\mathbb{Q})$ is the non-trivial involution sending  $\sqrt{d} \mapsto -\sqrt{d}$, and
    \item $\mathcal{O}_d = \begin{cases}
        \mathbb{Z}[\sqrt{d}] & d \equiv 2,3 \text{ mod }4 \\
        \mathbb{Z}[\frac{1+\sqrt{d}}{2}] & d \equiv 1\text{ mod }4 \\
    \end{cases}$ \hspace{0.5cm} is the ring of integers of $F$.\\
    
\end{itemize}
\vspace{0.1cm}
Let $J \in M_{m \times m}(F)$ be sesqui-symmetric matrix with respect to $\tau$ (i.e. $J^\top = \tau(J)$). We view $J$ as a $\tau$-Hermitian form on the $F$-vector space $F^m$. The associated special unitary group is

$$\text{SU}(J;F,\tau) = \{M \in SL(m,F):M^*JM=J\} \vspace{0.2cm}$$

\noindent where $M^* := \tau(M)^\top$. The integer points of this group form the associated \textit{integral special unitary group}
\vspace{-0.1cm}
$$\text{SU}(J;\mathcal{O}_d,\tau) = \{M \in SL(m,\mathcal{O}_d):M^*JM=J\} \vspace{0.2cm}.$$

\end{construction}

\begin{definition}
    $\tau$-Hermitian forms $J$ and $J'$ on $F^m$ are \textit{SU-equivalent} if $J'=P^*JP$ for some change of basis matrix $P \in GL(m,F)$.
\end{definition}

\noindent Observe that $$\text{SU}(P^*JP;F,\tau) = P^{-1}(\text{SU}(J;F,\tau))P .$$ The corresponding integral groups may not be conjugate, but $P \in GL(m,F)$ is in the commensurator of both integral groups, so they are commensurable lattices. Our interest in commensurability classes of arithmetic lattices with entries in $\mathcal{O}_d$ leads to the question: when are two $\tau$-Hermitian forms over $F^m$ equivalent? 
\vspace{0.15cm}
\noindent
\begin{definition} Let $F=\mathbb{Q}(\sqrt{d})$ and $\tau \in \text{Gal}(F/\mathbb{Q})$ be the involution sending $\sqrt{d} \mapsto - \sqrt{d}$.  A \textit{$\tau-$Hermitian square} is an element $g \in F$ such that $g=\tau(h)h$ for some $h \in F$.
\end{definition}

By Landherr \cite{Landherr1}, (see \cite{Lewis1} section 3) an equivalence class of $\tau$-Hermitian forms on $F^m$ for $m \geq 3$ is uniquely determined by  
\begin{itemize}
    \item the rank of the form, and
    \item the discriminant of the form up to $\tau-$Hermitian square.\footnote{Since $(\sqrt{d})\tau(\sqrt{d})=-d$, there is no signature in this setting.} 
\end{itemize}

\begin{proposition}\label{our lattices}
    \textit{Suppose $J$ is a full rank $\tau$-Hermitian form on $F^m$. Then the group $SU(J;\mathcal{O}_d,\tau)$ as constructed above is a non-uniform arithmetic lattice. Moreover, $SU(J;\mathcal{O}_d,\tau)$ is commensurable to $SU(J';\mathcal{O}_d,\tau)$ for any full rank $\tau-$Hermitian form $J'$ such that $|det(J) - det(J')|$ is a $\tau-$Hermitian square.}
\end{proposition}

\begin{proof}
    That $SU(J;\mathcal{O}_d,\tau)$ is an arithmetic lattice follows from Proposition (6.8.14) in \cite{Witte}. The same proposition tells us our lattice is non-uniform if and only if there exists a nonzero $x \in F^n$ such that $x^*Jx=0$. \\
    \indent By the classification above, $SU(J;\mathcal{O}_d,\tau)$ is commensurable to $SU(\text{diag}(1,-1,- \text{det}(J),1 \dots,1);\mathcal{O}_d,\tau)$. For $x=[1,1,0,\dots,0]$, clearly $$x^*\text{diag}(1,-1,-\text{det}(J),1, \dots,1) x = 0.$$ Thus $SU(J;\mathcal{O}_d,\tau)$ is non-uniform. 
\end{proof}

\bigskip

\section{Systolic Growth}

Next, we find a way to control the systole of certain lattices commensurable to those built by construction 1. There is a 1-1 correspondence between closed geodesics in $M_{\Gamma}$ and $\Gamma-$conjugacy classes of semi-simple elements in $\Gamma$. The length of the geodesic corresponding to the conjugacy class of a semi simple $\gamma \in \Gamma$ is proportional to the translation length of $\gamma$ on the geodesic it leaves invariant in $SL(m,\mathbb{R})/SO(m)$. Let $l(\gamma)$ denote this length. Hence $$\text{sys}(\Gamma) = \text{inf}\{l(\gamma): \text{ semi-simple }\gamma \in \Gamma\}.$$
\noindent The translation lengths are then bounded from below in terms of the trace (\cite{Lapan1}, theorem 3.1): 

\begin{theorem}{\textbf{(Trace-Length Bounds)}}
    \textit{Let $\gamma \in SL(m,\mathbb{R})$ be semi-simple  
    with $|\text{tr}(\gamma)| \geq 1$. Then
    $$l(\gamma) \geq \sqrt{2} \  \text{arccosh}\Big{(}\text{max}\Big{\{}1,\frac{|\text{tr}(\gamma)|}{m}\Big{\}}\Big{)}.$$}
\end{theorem}
Since $\lim_{x \rightarrow \infty}\text{arccosh}(x)=\infty$, we can control the lower bound for the systole of $M_\Gamma$ by controlling the lower bound for the traces of semisimple elements in $\Gamma$. Our tool for controlling the lower bound of the trace is principal congruence subgroups.

\begin{definition}
   Let $p$ be a rational prime and $\Gamma < SL(m,\mathbb{R})$ an arithmetic lattice with entries in a ring of integers $\mathcal{O}$. Then
   $$\Gamma^{(p)} = \text{Ker}(\pi_p:\Gamma  \rightarrow  SL(m, \mathcal{O}/(p)))$$
     is the \textit{principal congruence subgroup of $\Gamma$ of level $p$}, where $\pi_p$ is projection modulo $(p)$.
\end{definition}
\noindent $\Gamma^{(p)}$  is a normal subgroup of finite index in $\Gamma$. We will need the following proposition, whose proof uses ideas from the proof of Theorem 5.1 and Corollary 5.2 in \cite{Lapan1}.

\begin{proposition}\label{systole grows}
\textit{Fix square free $d \in \mathbb{Z}_{\geq 2}$. Let $F = \mathbb{Q}(\sqrt{d})$, $\mathcal{O}_d$ the ring of integers of $F$, and $\tau$ the non-trivial Galois automorphism of $F$ over $\mathbb{Q}$. Suppose $\{ J_n\}_{n \in \mathbb{N}} $ is a sequence of $\tau-$Hermitian forms over $F^{m}$.  For each $n$, let 
$$\Gamma_n = SU(J_n;\mathcal{O}_d ,\tau)$$
If $\{p_n\}_{n \in \mathbb{N}}$ is a sequence of rational primes diverging to $\infty$, then 
$$\text{sys}(\Gamma_n^{(p_n)})  \rightarrow \infty$$
where $\Gamma_n^{(p_n)}$ denotes the principal congruence subgroup of $\Gamma_n$ of level $p_n$.} 
\end{proposition}

\begin{proof} 
Let $k \in \mathbb{R}$. Since $\lim_{x \rightarrow \infty}\text{arccosh}(x)=\infty$, there exists $M \geq 2$ such that $$\frac{2\sqrt{2}}{m} \ \text{arccosh}(M-1) \geq k.$$ 


Since $p_n \rightarrow \infty$, there exists $N \in \mathbb{N}$ such that $p_n > mM$ for all $n \geq N$. Fix $n \geq N$ and suppose $\gamma \in \Gamma_n^{(p_n)}$ is semi-simple. Then for each power $q$ of $\gamma$, $\gamma^q$ is a semi simple element of $\Gamma_n^{(p_n)}$. Hence $\text{tr}(\gamma^q) \equiv m \ \text{mod} \ p_n$, so 
$$\text{tr}(\gamma^q) = p_n\alpha_q + m$$
for some $\alpha_q \in \mathcal{O}_d$. By the argument in the second paragraph\footnote{The argument in the aforementioned paragraph does not use that $\Gamma$ is derived from a central simple algebra, only that $\gamma$ is a semi-simple element and Newton's identities to obtain a formula for the characteristic polynomial of $\gamma$ in terms of the trace of powers of $\gamma$.}  of the proof of Theorem 5.1 in \cite{Lapan1}, there exists an integer $q$, $|q| \leq \frac{m}{2}$ such that $\text{tr}(\gamma^q) \neq m$. Set $\alpha := \alpha_q$. Since $\text{tr}(\gamma^q) \neq m$, $\alpha \neq 0$. Hence 
\begin{align*}|\text{tr}(\gamma^q)| & > m(M |\alpha|-1) \  \ \  \  \ \text{and} \\
|\tau(\text{tr}(\gamma^q))|  & > m(M |\tau(\alpha)|-1)
\end{align*}
For any $\alpha \in \mathcal{O}_d$, max$(|\alpha|, |\tau(\alpha)|) \geq 1$. Thus
$$\text{max}\Big{\{}\frac{|\text{tr}(\gamma^q)|}{m},\frac{|\tau(\text{tr}(\gamma^q))|}{m} \Big{\}}> M-1.$$

\noindent By definition of the special unitary group, $\gamma^*=J_n\gamma^{-1}J_n^{-1}$. Hence  $$\tau(\text{tr}(\gamma^q))= \text{tr}((\gamma^q)^*) = \text{tr}(\gamma^{-q}).$$
Since $l(\gamma^q)=l(\gamma^{-q})$, Theorem 3 implies
$$l(\gamma^q) \geq \sqrt{2} \ \text{arccosh}\Big{(}\text{max}\Big{\{}1,\frac{|\text{tr}(\gamma^q)|}{m},\frac{|\tau(\text{tr}(\gamma^q))|}{m}\Big{\}}\Big{)}.$$
\noindent Since arccosh is increasing on $[1,\infty)$ and $M \geq 2$, 
$$l(\gamma^q) \geq \sqrt{2}  \  \text{arccosh}\Big{(} M-1\Big{)} .$$
Since $l(\gamma^q) = |q| l (\gamma)$ and $|q| \leq \frac{m}{2}$, 
$$l(\gamma) \geq \frac{2\sqrt{2}}{m} \  \text{arccosh}\Big{(}M-1\Big{)} \geq k .$$
Hence 
$$\text{sys}(\Gamma_n^{(p_n)}) \geq k.$$
Since $k$ is arbitrary, this completes the proof. 
\end{proof}

\bigskip

\section{Result}

\noindent \textbf{Theorem 1.}
    \textit{Fix square free $d \in \mathbb{Z}_{\geq 2}$ and let
    $$\mathcal{L} = SU(I_{8};\mathcal{O}_d,\tau). $$
    There exists a sequence of non-uniform arithmetic lattices $\Lambda_n < SL(8,\mathbb{R})$ commensurable to $\mathcal{L}$ such that 
    $$\text{sys}(\Lambda_n) \rightarrow \infty$$
    as $n \rightarrow \infty$, yet every $\Lambda_n$ contains a fixed hyperbolic 3-manifold group $\Pi$, a finite index subgroup of vol3.}

\vspace{0.1cm}
\noindent \textbf{Remark.} The 3-manifold group $\Pi$ contains surface groups, so each $\Lambda_n$ contains a fixed surface group. In particular, the systolic genus of these lattices is bounded from above for all $n$. \\

Vol3 is an arithmetic hyperbolic 3-manifold with the third lowest volume in the census. The fundamental group of vol3, which we will also refer to as vol3, has presentation 

$$\text{vol3}=\langle a,b \ | \ aabbABAbb; \ aBaBabaaab \rangle \vspace{0.15cm}$$

\noindent  where $A = a^{-1}$ and $B = b^{-1}$. The hyperbolic representation of vol3 into $SO(3,1)$ admits discrete and faithful deformations in $SL(4,\mathbb{R})$. An explicit one-parameter family of these deformations was found by Cooper, Long, and Thistlethwaite \cite{Cooper1}. Vol3 covers an orbifold, denoted vol3$/\langle u \rangle$, which has a simpler representation
$$\rho_t:\text{vol3}/\langle u \rangle \rightarrow SL(4, \mathbb{Q}(t,\sqrt{t^2-1},\sqrt{t^2+2})) \vspace{0.15cm} $$
for $t \geq 1$. The representation in \cite{Cooper1} uses parameter $v$ instead of $t$, with the substitution $v=2t$. Two elements, denoted here $u$ and $c$, generate vol3$/\langle u \rangle$, and the image of these elements under $\rho_t$ are listed in the appendix, as well as in the accompanying mathematica file \cite{Hillen}. To recover the manifold group, one can use the relations $a=u^2c$ and $b=(aua)^{-1}u$. In practice, we work with the orbifold group when interacting with the explicit matrix representation.
 
 The hyperbolic representation is at $t=1$ and for real values $t\geq 1$ the representation is the holonomy of a real projective structure on vol3, and is thus discrete and faithful.  It is not clear how to specialize $t$ to ensure the entries of the image all lie in a ring of integers over some field. As luck would have it, we found 16 specific elements $\{1,g_1,\dots,g_{15}\} \in$ vol3$/\langle u \rangle $ such that

\begin{itemize}
    \item $\mathcal{B}_v = \{\rho_v(1),\rho_v(g_1),\dots,\rho_v(g_{15})\}$ is a basis for the vector space $M_{4 \times 4}(\mathbb{R})$.
\item The left regular representation of vol3$/\langle u \rangle$ with respect to the basis $\mathcal{B}_v$ yields a representation 
$$\eta_{t}: \text{vol3}/\langle u \rangle \rightarrow SL(16,\mathbb{Q}(t,\sqrt{t^2-1}))$$
\item Both $\eta_t(u)$ and $\eta_t(c)$ have entries in $\mathbb{Z}[t,\sqrt{t^2-1}]$
\end{itemize}

In an attempt to find an integral representation of smaller dimension, we study invariant subspaces of this 16 dimensional representation. By considering eigenspaces of elements in the centralizer of $\eta_t(\text{vol3}/\langle u \rangle)$, we were able to find an 8 dimensional invariant subspace, whose corresponding representation has entries in $\mathbb{Z}[t, \sqrt{t^2-1}]$. Let 
$$\omega_t : \text{vol3}/\langle u \rangle \rightarrow SL(8,\mathbb{Z}[t,\sqrt{t^2-1}])$$
denote this representation. 
The matrices $\omega_t(u)$ and $\omega_t(c)$ are listed in the appendix. This representation is conjugate to $\rho_t \oplus \rho_t$, and hence faithful. Computations confirming $\omega_t$ is conjugate to $\rho_t \oplus \rho_t$, the 16 dimensional representation $\eta_t$, as well as the the explicit 8 dimensional subspace mentioned above can all be found in the accompanying mathematica file \cite{Hillen}. 

\bigskip

\section{Proof}

\noindent For the remainder of this paper, fix square free $d \in \mathbb{Z}_{\geq 2}$ and let 
$$\mathcal{L} := SU(I_{8};\mathcal{O}_d,\tau). $$

\vspace{0.2cm}
Interpreting $t$ as transcendental, let $\tau \in \text{Gal}(\mathbb{Q}(t,\sqrt{t^2-1})/\mathbb{Q}(t))$ be the involution sending $\sqrt{t^2-1} \mapsto -\sqrt{t^2-1}$. For a $\tau-$Hermitian form $J_t \in (\mathbb{Q}(t, \sqrt{t^2-1}))^m$, let
$$SU(J_t;\mathbb{Q}(t,\sqrt{t^2-1}),\tau) := \{M \in SL(m,\mathbb{Q}(t))) \ | \ M^*J_tM=J_t\} $$
for $M^* := \tau(M)^{\top}$. \\

We start the proof of Theorem 1 by finding a sequence of arithmetic lattices commensurable to $\mathcal{L}$ which contain vol3. 
\begin{lemma}\label{lemma 1}
   \textit{ There exists a family of Hermitian forms $J_t \in SL(8,\mathbb{Q}(t))$ such that 
    \begin{itemize}
        \item For $t \geq 1$, $J_t$ is full rank with det$(J_t)$ equal to a square in $\mathbb{Q}(t)$, and 
        \item $\omega_t(\text{vol3}) < SU(J_t;\mathbb{Q}(t,\sqrt{t^2-1}),\tau)$.
    \end{itemize}
    Moreover, there exists a sequence $t_n \rightarrow \infty$ such that 
    $$\omega_{t_n}(\text{vol3}) < SU(J_{t_n};\mathcal{O}_d,\tau)$$
    for all $n \in \mathbb{N}$. By Proposition \ref{our lattices}, the $SU(J_{t_n};\mathcal{O}_d,\tau)$ are commensurable to $\mathcal{L}$ for $n >>0$. }
\end{lemma}

\begin{proof}
    The first part of the lemma follows from a computation: We solve for  $J=J_t \in GL(8,\mathbb{R})$ such that $\omega_t(u)^*J\omega_t(u)=J$ and $\omega_t(c)^*J\omega_t(c)=J$. By replacing $J$ with $J+J^*$, we can ensure $J$ is sesqui-symmetric. There are 4 free variables in the solution for $J$, and by making a choice of numerical value for each free variable\footnote{Even leaving all four free variables in $J$ as unknowns, the determinant of $J$ is a square in $\mathbb{Q}(t)$. Thus choosing values in $\mathbb{Q}(t)$ for the free variables does not change the resulting commensurability class of the lattice.} we obtain a $\tau$-Hermitian form $J$ which is full rank (for all but finitely many choices of $t$) . Indeed,
    $$\text{det}(J_t) = \frac{16 (3 - 4 t^2)^4}{(1 - 4 t^2)^2}$$
    which is a square in $\mathbb{Q}(t)$.
    Moreover, for $t\geq 1$, det$(J_t)$ is nonzero, so $J_t$ is full rank. The matrix $J_t$ can be found in the accompanying mathematica file \cite{Hillen}.

    To guarantee $\omega_t($vol3$)$ lies in an \textit{integral} special unitary group, is it necessary for the entries of $\omega_t(a)$ and $\omega_t(b)$ to lie in $\mathcal{O}_d$. It is sufficient to choose $t \in \mathbb{N}$ so that $\sqrt{t^2-1} \in\mathcal{O}_d$. Equivalently, we need to find infinitely many integral solutions $(t,y)$ to Pell's equation:
    $$t^2 - d y^2 = 1. $$
It is well known that for any positive non-square $d \in \mathbb{Z}$, Pell's equation has a fundamental solution $(t_1,y_1) \in \mathbb{N}^2$ and the other solutions are exactly the integers $(t_n,y_n)$ such that
    $$u^n = t_n + y_n\sqrt{d}$$
    for $u=t_1+y_1\sqrt{d}$. Therefore, for this sequence $\{t_n\}_{n=1}^\infty$, 
    \vspace{0.2cm}
    $$\omega_{t_n}(vol3) < SU( J_{t_n} ;\mathcal{O}_{d} ,\tau )  $$ for all $n \in \mathbb{N}$. We can write
    $t_n = \frac{1}{2}(u^n + u^{-n})$, so the sequence $t_n \rightarrow \infty$ as $n \rightarrow \infty$. 
\end{proof}

Let $\Gamma_n := SU( J_{t_n} ;\mathcal{O}_{d} ,\tau )$ for the sequence $\{t_n\}_{n \in \mathbb{N}}$ from Lemma \ref{lemma 1}. Our next goal is to find finite index subgroups of $\Gamma_n$ whose systole goes to infinity as $n \rightarrow \infty$. This is accomplished by considering the principal congruence subgroups described in Section 3. We will let
$$\Lambda_{n}:=\Gamma_n^{(p_n)}$$ 
for carefully chosen primes $p_n \rightarrow \infty$. By Proposition \ref{systole grows},  sys$(\Lambda_n) \rightarrow \infty$ as $n \rightarrow \infty$. \\

We will choose primes $p_n$ so that $\omega_{t_n}(\Pi) < \Lambda_{n}$ for $\Pi < $ vol3 a finite index subgroup. More specifically, set
$$\Pi = \text{Ker}(\omega_0:\text{vol3} \rightarrow SL(8,\mathbb{Z}[i]))\vspace{0.15cm}.$$
At $t=0$, one can check that $|\omega_0(\text{vol3})|=320$ computationally, so indeed $\Pi$ is finite index in vol3.\\

Consider the following diagram:

\vspace{0.5cm}

\begin{center}

\tikzset{every picture/.style={line width=0.75pt}} 

\begin{tikzpicture}[x=0.75pt,y=0.75pt,yscale=-1,xscale=1]

\draw    (105,123) -- (105,171) ;
\draw [shift={(105,173)}, rotate = 270] [color={rgb, 255:red, 0; green, 0; blue, 0 }  ][line width=0.75]    (6.56,-2.94) .. controls (4.17,-1.38) and (1.99,-0.4) .. (0,0) .. controls (1.99,0.4) and (4.17,1.38) .. (6.56,2.94)   ;
\draw    (123,190) -- (175,190.48) ;
\draw [shift={(177,190.5)}, rotate = 180.53] [color={rgb, 255:red, 0; green, 0; blue, 0 }  ][line width=0.75]    (6.56,-2.94) .. controls (4.17,-1.38) and (1.99,-0.4) .. (0,0) .. controls (1.99,0.4) and (4.17,1.38) .. (6.56,2.94)   ;
\draw    (128,110) -- (332,110) ;
\draw [shift={(334,110)}, rotate = 180] [color={rgb, 255:red, 0; green, 0; blue, 0 }  ][line width=0.75]    (6.56,-2.94) .. controls (4.17,-1.38) and (1.99,-0.4) .. (0,0) .. controls (1.99,0.4) and (4.17,1.38) .. (6.56,2.94)   ;
\draw [shift={(128,110)}, rotate = 0] [color={rgb, 255:red, 0; green, 0; blue, 0 }  ][line width=0.75]      (0,-6.71) .. controls (-1.85,-6.71) and (-3.35,-5.21) .. (-3.35,-3.35) .. controls (-3.35,-1.5) and (-1.85,0) .. (0,0) ;
\draw  [dash pattern={on 4.5pt off 4.5pt}]  (126,29.5) -- (325,30) ;
\draw [shift={(327,30)}, rotate = 180.14] [color={rgb, 255:red, 0; green, 0; blue, 0 }  ][line width=0.75]    (7.65,-3.43) .. controls (4.86,-1.61) and (2.31,-0.47) .. (0,0) .. controls (2.31,0.47) and (4.86,1.61) .. (7.65,3.43)   ;
\draw [shift={(126,29.5)}, rotate = 0.14] [color={rgb, 255:red, 0; green, 0; blue, 0 }  ][line width=0.75]      (0,-7.83) .. controls (-2.16,-7.83) and (-3.91,-6.07) .. (-3.91,-3.91) .. controls (-3.91,-1.75) and (-2.16,0) .. (0,0) ;
\draw    (383,50.5) -- (383,92) ;
\draw [shift={(383,94)}, rotate = 270] [color={rgb, 255:red, 0; green, 0; blue, 0 }  ][line width=0.75]    (6.56,-2.94) .. controls (4.17,-1.38) and (1.99,-0.4) .. (0,0) .. controls (1.99,0.4) and (4.17,1.38) .. (6.56,2.94)   ;
\draw [shift={(383,50.5)}, rotate = 90] [color={rgb, 255:red, 0; green, 0; blue, 0 }  ][line width=0.75]      (0,-6.71) .. controls (-1.85,-6.71) and (-3.35,-5.21) .. (-3.35,-3.35) .. controls (-3.35,-1.5) and (-1.85,0) .. (0,0) ;
\draw    (384,127) -- (384,171) ;
\draw [shift={(384,173)}, rotate = 270] [color={rgb, 255:red, 0; green, 0; blue, 0 }  ][line width=0.75]    (6.56,-2.94) .. controls (4.17,-1.38) and (1.99,-0.4) .. (0,0) .. controls (1.99,0.4) and (4.17,1.38) .. (6.56,2.94)   ;
\draw    (105,46) -- (105,94) ;
\draw [shift={(105,96)}, rotate = 270] [color={rgb, 255:red, 0; green, 0; blue, 0 }  ][line width=0.75]    (6.56,-2.94) .. controls (4.17,-1.38) and (1.99,-0.4) .. (0,0) .. controls (1.99,0.4) and (4.17,1.38) .. (6.56,2.94)   ;
\draw [shift={(105,46)}, rotate = 90] [color={rgb, 255:red, 0; green, 0; blue, 0 }  ][line width=0.75]      (0,-6.71) .. controls (-1.85,-6.71) and (-3.35,-5.21) .. (-3.35,-3.35) .. controls (-3.35,-1.5) and (-1.85,0) .. (0,0) ;
\draw  [draw opacity=0] (257.87,145.59) .. controls (258.47,147.77) and (258.2,150.21) .. (256.92,152.35) .. controls (254.41,156.58) and (248.93,158.08) .. (244.68,155.7) .. controls (240.44,153.33) and (239.04,147.98) .. (241.56,143.76) .. controls (243.83,139.95) and (248.51,138.35) .. (252.52,139.82) -- (249.24,148.05) -- cycle ; \draw   (257.87,145.59) .. controls (258.47,147.77) and (258.2,150.21) .. (256.92,152.35) .. controls (254.41,156.58) and (248.93,158.08) .. (244.68,155.7) .. controls (240.44,153.33) and (239.04,147.98) .. (241.56,143.76) .. controls (243.83,139.95) and (248.51,138.35) .. (252.52,139.82) ;  
\draw    (257.89,145.65) -- (262,149) ;
\draw    (257.89,145.65) -- (254.71,149.59) ;
\draw    (280,188) -- (333,188) ;
\draw [shift={(335,188)}, rotate = 180] [color={rgb, 255:red, 0; green, 0; blue, 0 }  ][line width=0.75]    (6.56,-2.94) .. controls (4.17,-1.38) and (1.99,-0.4) .. (0,0) .. controls (1.99,0.4) and (4.17,1.38) .. (6.56,2.94)   ;

\draw (89,101) node [anchor=north west][inner sep=0.75pt]   [align=left] {$\displaystyle vol3$};
\draw (99,19) node [anchor=north west][inner sep=0.75pt]  [font=\normalsize] [align=left] {$\displaystyle \Pi $};
\draw (239,92.4) node [anchor=north west][inner sep=0.75pt]  [font=\small]  {$\omega _{t}$};
\draw (83,136.4) node [anchor=north west][inner sep=0.75pt]  [font=\small]  {$\omega _{0}$};
\draw (340,102.4) node [anchor=north west][inner sep=0.75pt]  [font=\normalsize]  {$SU( J_{t} ,\mathcal{O}_{d} ,\tau )$};
\draw (340,180.4) node [anchor=north west][inner sep=0.75pt]  [font=\normalsize]  {$SL( 8,\mathcal{O}_{d} /( p))$};
\draw (332,20.4) node [anchor=north west][inner sep=0.75pt]  [font=\normalsize]  {$\text{Ker}\left( \pi _{p} :SU( J_{t} ,\mathcal{O}_{d} ,\tau ) \ \rightarrow SL( 8,\mathcal{O}_{d} /( p))\right)$};
\draw (145,170.4) node [anchor=north west][inner sep=0.75pt]  [font=\small]  {$\pi _{p}$};
\draw (50,180.4) node [anchor=north west][inner sep=0.75pt]  [font=\normalsize]  {$SL( 8,\mathbb{Z}[ i])$};
\draw (390,136.4) node [anchor=north west][inner sep=0.75pt]  [font=\small]  {$\pi _{p}$};
\draw (185,179.4) node [anchor=north west][inner sep=0.75pt]  [font=\normalsize]  {$SL( 8,\mathbb{Z}[ i] /( p))$};
\draw (310,167.4) node [anchor=north west][inner sep=0.75pt]  [font=\small]  {$f_{*}$};

\end{tikzpicture}

\vspace{0.5cm}

\end{center}

\noindent The primes we choose will divide $t$. Since our choice of $t$ solves Pell's equation, we have $(pk)^2-1=dy^2$ for integers $k$ and $d$. Thus
$$f:\mathbb{Z}[i]/(p) \rightarrow \mathcal{O}_d/(p)$$
induced by sending $\overline{1} \mapsto \overline{1}$ and $\overline{i} \mapsto \overline{y\sqrt{d}}$ is a ring homomorphism. This induces a group homomorphism $$f_*:SL(8,\mathbb{Z}[i]/(p)) \rightarrow SL(8,\mathcal{O}_d/(p)).$$
\noindent Our purpose in constructing this diagram is in the observation that, as long as the diagram commutes, we can guarantee $\omega_t(\Pi) < \text{Ker}(\pi_p)$. In fact, by inspecting the representation (see Appendix), $p$ dividing $t$ is sufficient to guarantee the diagram commutes. \\

\begin{lemma}\label{lemma 2}
    \textit{Let $\{t_n\}$ be as in Lemma 1. Possibly passing to a subsequence of $t_n$, there exists a sequence of primes $p_n \rightarrow \infty$ with each $p_n$ dividing $t_n$. Hence 
    $$\omega_{t_n}(\Pi) < \Gamma_{n}^{(p_n)}$$ for all $n \in \mathbb{N}$. }

\end{lemma}

\begin{proof}
   For each $n$, we would like a prime $p_n$ dividing $t_n$, but not dividing $t_{m}$ for $m<n$. The latter condition will ensure a subsequence of $p_n$ is strictly increasing.  Some results from number theory will help us accomplish this. We start with the following definition from \cite{BHV}: 

    \begin{definition}
        \textit{A pair of algebraic integers $(\alpha,\beta)$ is called a \textbf{Lucas pair} if $\alpha+\beta$ and $\alpha \beta$ are non-zero coprime rational integers and $\frac{\alpha}{\beta}$ is not a root of unity.}
    \end{definition}

     By \cite{Ca}, (or see Theorem A in \cite{BHV}) if $(\alpha,\beta)$ is a Lucas pair, then for $n >>0 $ the $n$-th term of the sequence
    $$S_n :=\alpha^n + \beta^n$$
    has a primitive prime divisor, i.e. a prime $p_n$ such that $p_n$ divides $S_n$ but not $S_m$ for $m<n$. \\

    \noindent \textit{Claim: Let $u=t_1 + y_1\sqrt{d}$ be as in the proof of Lemma 1. Then $(u,1/u)$ is a Lucas pair.} \\

    \noindent Observe that
    $$2t_n = u^n+\Big{(}\frac{1}{u}\Big{)}^n$$
    Therefore, pending the claim above, the sequence $2t_n$ has a corresponding sequence of primitive prime divisors, $p_n$ for $n>>0$. Passing to a subsequence, we may assume the $p_n$ are strictly increasing, each prime $p_n \neq 2$. Hence each $p_n$ divides $t_n$ and $p_n \rightarrow \infty$.  Hence the diagram above commutes and therefore
    $$\omega_{t_n}(\Pi) < \Gamma_n^{(p_n)}.$$

    \noindent \textit{Proof of Claim:} Since $u \in \mathbb{Z}[\sqrt{d}] \subseteq \mathcal{O}_d$, we know $u$ is an algebraic integer. Since $t_1^2-dy_1^2=1$,
    \begin{align*}
        \frac{1}{u} & = t_1 - y_1\sqrt{d}
    \end{align*}
    Hence $\frac{1}{u}$ is also an algebraic integer. Further, $u +\frac{1}{u} = 2t_1 $ and $u (\frac{1}{u}) = 1$, so $u + \frac{1}{u}$ and $u(\frac{1}{u})$ are non-zero coprime rational integers. Moreover $\frac{u}{1/u} = u^2$ is not a root of unity. Hence $(u,\frac{1}{u})$ is a Lucas pair. 
  \end{proof}
 
\vspace{-0.1cm}
\noindent Now we put these pieces together to prove the main theorem:\\

\noindent \textbf{Theorem 1.}
    \textit{Fix square free $d \in \mathbb{Z}_{\geq 2}$ and let
    $$\mathcal{L} := SU(I_{8};\mathcal{O}_d,\tau). $$
    There exists a sequence of non-uniform arithmetic lattices $\Lambda_n < SL(8,\mathbb{R})$ commensurable to $\mathcal{L}$ such that 
    $$\text{sys}(\Lambda_n) \rightarrow \infty$$
    as $n \rightarrow \infty$, yet every $\Lambda_n$ contains a fixed hyperbolic 3-manifold group $\Pi$, a finite index subgroup of vol3.}\\

\begin{proof}
    For each $n \in \mathbb{N}$, let $\Gamma_n := SU(J_{t_n};\mathcal{O}_d,\tau)$ for forms $J_{t_n}$ from Lemma 1 and let
    $$\Lambda_{n} := \Gamma_n^{(p_n)}$$
    be the principal congruence subgroup of $\Gamma_n$ of level $p_n$  for primes $p_n$ from Lemma 2. By Lemma 1, each $\Gamma_n$ is a non-uniform lattice commensurable to $\mathcal{L}$. Since $\Lambda_{n}$ is a finite index subgroup of $\Gamma_n$, each $\Lambda_n$ is also a non-uniform lattice commensurable to $\mathcal{L}$. By Lemma 2, $\omega_{t_n}(\Pi) < \Lambda_n$  and by Proposition \ref{systole grows}, sys$(\Lambda_{n}) \rightarrow \infty$ as $n \rightarrow \infty$.  
\end{proof}

 \bigskip

\section{Examples}

\begin{example} We describe how to find the first few terms of the sequence $(t_n, p_n)$ in the case that $d=3$. Since $3 \not\equiv 1 $(mod $4$), we have $\mathcal{O}_3 = \mathbb{Z}[\sqrt{3}]$. \\

To guarantee the entries of $\omega_t(u)$ and $\omega_t(c)$ lie in $\mathcal{O}_3$, we solve for $t$ in Pell's equation:
$$t^2-3y^2 = 1.$$

The fundamental solution is $(t_1,y_1)=(2,1)$. 
Let $u = 2+\sqrt{3}$. Powers of $u$ allow us to find all the other solutions. To ensure the diagram commutes and consequently that $\omega_t(\Pi) < \Gamma_n^{(p)}$, we need to find corresponding primes $p_n$ such that $t_n = 0 \text{ mod } p_n$. \\
\begin{table}[H]
    \centering
    \begin{tabular}{|c|l|l|l|}
    \hline
        $n$ & $u^n$& $t_n$ & $p_n$ \\
        \hline
        $1$ & $2+\sqrt{3}$ & $2$ & $2$\\
        $2$ & $7+4\sqrt{3}$ & $7$ & $7$ \\
        $3$ &  $26 + 15 \sqrt{3}$ & $26$& $13$ \\
        $4$ & $97 + 56 \sqrt{3}$ & $97$ & $97$ \\
        $5$ & $362 + 209 \sqrt{5}$ & $362$ & $181$ \\
        \hline
    \end{tabular}
    \label{tab:my_label}
\end{table}
Observe each prime does not divide any of the previous values of $t_n$. \\
\end{example}

\begin{example}
Now let $d=5$. Since $5 \equiv 1 $ (mod $4$), $\mathcal{O}_5 = \mathbb{Z}[\frac{1+\sqrt{5}}{2}]$. We first solve Pell's equation:
$$t^2 - 5y^2 = 1.$$ 
The fundamental solution is $(t_1,y_1)=(9,4)$. Following the same process as above, we find the first 5 terms of the sequences $t_n$ and $p_n$. \\

\begin{table}[H]
    \centering
    \begin{tabular}{|c|l|l|l|}
    \hline
        $n$ & $u^n$& $t_n$ & $p_n$ \\
        \hline
        $1$ & $9+4\sqrt{5}$ & $9$ & $3$\\
        $2$ & $161+72\sqrt{5}$ & $161$ & $7$ \\
        $3$ &  $2889 + 1292 \sqrt{5}$ & $2889$& $107$ \\
        $4$ & $51841 + 23184 \sqrt{5}$ & $51841$ & $1103$ \\
        $5$ & $930249 + 416020 \sqrt{5}$ & $930249$ & $2521$ \\
        \hline
    \end{tabular}
    \label{tab:my_label}
\end{table}
    
\end{example}
\bigskip

\section{Further Questions and Observations}

\begin{enumerate}
   
  \item \textit{Thin 3-manifold Group:} Our 8 dimensional representation $\omega$ is conjugate to two copies of the original representation $\rho$. By the proof of Theorem 2.1 in \cite{Long2}, $\rho_t(\text{vol3})$ is Zariski dense in $SL(4,\mathbb{R})$. Therefore the Zariski closure of $\omega_t($vol3$)$ in $SL(8,\mathbb{R})$ is a subgroup isomorphic to $SL(4,\mathbb{R})$. 
  We also observe that for any of our chosen values of $t$, $\omega_{t}(\Pi)$ is infinite index in $SU(J_{t};\mathcal{O}_d,\tau)$. To see this, observe that $SU(J_{t};\mathcal{O}_d,\tau)$ is an irreducible arithmetic lattice of $\mathbb{R}-$rank $\geq 2$. Thus, by Margulis's Normal Subgroup Theorem, if $\omega_{t}($vol3$)$ were finite index in $SU(J_{t};\mathcal{O}_d,\tau)$, then the abelianization of $\omega_{t}($vol3$)$ would be finite. However, vol3 is an arithmetic hyperbolic 3-manifold with positive first Betti number. Thus vol3 must have infinite virtual first Betti number, so the abelianization cannot be finite (see \cite{Cooper2}).
  It would be interesting to know whether $\omega_t($vol3$)$ is infinite index in the intersection of $SU(J_{t};\mathcal{O}_d,\tau)$ and Zcl$(\omega_t($vol3$)) \cong SL(4,\mathbb{R})$.  If so, then  $\omega_t($vol3$)$ is a thin 3-manifold group. \\
   
    \item \textit{Non-Hitchin Surface Subgroups:} At $v=2$ the image $\rho_2(\text{vol3})$ lies in $SO(3,1)$, so every non-trivial semi-simple element has a pair of eigenvalues with the same modulus. Thus $\rho_2$ is not Borel Anosov (for example, by the characterization of Anosov in theorem 4.3 of \cite{Kassel}). Moreover, the restriction to any subgroup is also not Borel Anosov. Hence for any value of $v \geq 2$ and any surface subgroup $H$, $\rho_v(H)$ is a discrete and faithful representation in $SL(4,\mathbb{R})$ which is not on the Hitchin component. This is notable, since many examples of interesting representations of surface subgroups in the literature are in the Hitchin component. In particular, if one can find a surface subgroup of $\rho_v($vol3$)$ which is Borel Anosov for some value of $v>2$, then it would be an example of a non-Hitchin Borel Anosov representation, answering a question posed by Canary \cite{Canary}, Question 50.6. 
\end{enumerate}
\bigskip

\section{Appendix}
\subsection{The original 4 dimensional representation of vol3}The hyperbolic 3-manifold vol3 has presentation
$$\text{vol3}=\langle a,b \ | \ aabbABAbb; \ aBaBabaaab \rangle \vspace{0.1cm}$$
\noindent  where $A = a^{-1}$ and $B = b^{-1}$. The orbifold group vol3 $/ \langle u \rangle$ which contains vol3 as a subgroup is generated by $u$ and $c$, of order 4 and 2 respectively. To recover vol3, we have $a=u^2c$ and $b=(aua)^{-1}u$. A conjugate of the image of $u$ and $c$ under the original 4D representation from section 2.4 of \cite{Cooper3} with the substitution $t = \frac{v}{2}$ are: \\

\vspace{0.5cm}

\begin{adjustwidth}{-0.1cm}{0pt}
$\rho_t(u)= \left(
\begin{array}{cccc}
1 & 0 & 0 & 0\\
0& 1& 0& 0 \\
0& 0& \sqrt{\frac{(t^2-1)}{(2 + 
     t^2)}}& 1\\
0& 0& -\frac{(1 + 2 t^2)}{(2 + 
      t^2)} & -\sqrt{\frac{( t^2-1)}{(2 + t^2)}}
\end{array}\right)$

\end{adjustwidth}

\vspace{1cm}

\begin{adjustwidth}{-0.1cm}{0pt}
$\rho_t(c)=\left(
\begin{array}{cccc}
\frac{1}{2} (t + \sqrt{2 + t^2})& 0& \frac{1}{2} (1 - t^2 - t \sqrt{2 + t^2})& 0\\
0&  \frac{1}{2}(t - \sqrt{2 + t^2})& 0& \frac{1}{2} (-1 + t^2 - t \sqrt{2 + t^2})\\
1&0 & \frac{1}{2}(-t - \sqrt{2 + t^2})& 0\\
0& -1& 0&  \frac{1}{2} (-t + \sqrt{2 + t^2})
\end{array}\right)$

\end{adjustwidth}
\vspace{0.5cm}

The exact representation from \cite{Cooper3} and the matrix which conjugates it to $\rho_t$ listed above can be found in the accompanying mathematica file \cite{Hillen}. \\

Let $\tau \in \text{Gal}(\mathbb{Q}(t,\sqrt{t^2-1})/\mathbb{Q}(t))$ sending $\sqrt{t^2-1}$ to $-\sqrt{t^2-1}$. There is one $\tau-$Hermitian form (up to scalar multiples) which both generators $\rho_t(u)$ and $\rho_t(c)$ preserve: \\
\vspace{0.25cm} 

\begin{adjustwidth}{-0.1cm}{0pt}
$M_t=\left(
\begin{array}{cccc}
-\frac{(2\sqrt{2 + t^2})}{(
  2 t + t^3 - \sqrt{2 + t^2} + t^2 \sqrt{2 + t^2})}& 0& 0& 0\\ 0& \frac{2 (2 + t^2)}{((1 + 2 t^2) (1 - t^2 + 
    t\sqrt{2 + t^2}))}& 0& 0\\
    0& 0& 1& 0 \\
    0& 0& 0& \frac{(2 + t^2)}{(
 1 + 2 t^2)}
\end{array}\right)$
\end{adjustwidth}
\vspace{0.5cm}
Note that for $t=1$, we have $\sqrt{t^2-1}=0$, so the involution $\tau$ is trivial. Since $M_1$ is a diagonal matrix with signature $(3,1)$, it is clear that $\rho_1$ lies in $SO(3,1)$. \\
\vspace{0.5cm}
\smallskip

\subsection{The explicit 8 dimensional representation of vol3:} Our 8D representation which is integral for values of $t$ solving Pell's equation is generated by:
\vspace{0.5cm}

\begin{adjustwidth}{-0.1cm}{0pt}
$\omega_t(u)= \left(
\begin{array}{cccccccc}
1& 0& 0& 2 t& -2 t& 2 t& 0& 0\\
0& 1& 0& t - \sqrt{t^2-1}& -t + \sqrt{t^2-1}& t - \sqrt{t^2-1}& 1& 0\\
0& 0& 1& -1& 1& -1& 0& 0\\
0& 0& 1& 0& 0& 0& 0& 0\\
0& 0& 1&  0& 0& 1& -t + \sqrt{t^2-1}& -t -\sqrt{t^2-1}\\
0& 0& 1& -1& 0&
   1& -t + \sqrt{t^2-1}& -t - \sqrt{t^2-1}\\
0& 0& 0& 0& 0& 0& 0& -1\\
0& 0& 0& 0& 0& 0& 1& 0
\end{array}\right)$

\end{adjustwidth}

\vspace{1cm}

\begin{adjustwidth}{-0.1cm}{0pt}
$\omega_t(c)=\left(
\begin{array}{cccccccc}
0 & 0 & 1 & 0 & 0 & 0 & 0 & 0\\
0 & 0 & 0 & 1 & 0 & 0 & 0 & 0\\ 
1 & 0 & 0 & 0 & 0 & 0 & 0 & 0\\ 
0 & 1 & 0 & 0 & 0 & 0 & 0 & 0\\
0 & 0 & 0 & 0 & 0 & 0 & -1 & 1\\ 
0 &  0 & 0 & 0 & 0 & 0 & -1 & 0\\ 
0 & 0 & 0 & 0 & 0 & -1 & 0 & 0\\ 
0 & 0 & 0 & 0 & 1 & -1 & 0 & 0

\end{array}\right)$

\end{adjustwidth}
\vspace{0.5cm}
An explicit computation found in \cite{Hillen} shows this representation is conjugate to $\rho_t \oplus \rho_t$. One can also find the $\tau$- Hermitian form $J_t$ which both $\omega_t(u)$ and $\omega_t(c)$ preserve in the mathematica file in \cite{Hillen}.

\vspace{1cm}
\noindent Department of Mathematics,\\ University of California,\\ Santa Barbara, 
CA 93106.\\
\noindent Email:~paigehillen@math.ucsb.edu\\[\baselineskip]
\end{document}